%% file: LS.tex
\documentclass[11pt,english]{amsart}

\input{packages_and_functions.for_preamble.tex}

\usepackage{needspace}
\definecolor{articlelink}{rgb}{0,0,0.45}
\hypersetup{linkcolor=articlelink,citecolor=articlelink,urlcolor=articlelink,
  pdftitle={The Lange--Stuhler Theorem via Gerbes},
  pdfauthor={Lei Zhang}}

\theoremstyle{plain}
\newtheorem{thm}{Theorem}[section]
\newtheorem{lem}[thm]{Lemma}
\newtheorem{prop}[thm]{Proposition}
\newtheorem{cor}[thm]{Corollary}
\newtheorem{thmI}{Theorem}

\theoremstyle{remark}
\newtheorem{rmk}[thm]{Remark}

\input{packages_and_functions.tex}

\newcommand{\Fp}{\mathbb F_p}
\newcommand{\Fq}{\mathbb F_q}
\newcommand{\Eq}{\operatorname{Eq}}
\newcommand{\Fixst}{\operatorname{Fix}}

\begin{document}

\title{The Lange--Stuhler Theorem via Gerbes}
\author{Lei Zhang}
\email{cumt559@gmail.com}
\address{Sun Yat-Sen University\\ School of Mathematics (Zhuhai)\\
Zhuhai, Guangdong Province\\ China}
\date{September 2026}

\begin{abstract}
    We recast the Lange–Stuhler theorem as the universal property of a \(2\)-cartesian square of classifying stacks. This formulation shows that every Frobenius-periodic vector bundle of rank \(r\) on an arbitrary fibered category in characteristic \(p>0\) is trivialized by a finite étale torsor, without geometric hypotheses or restrictions on the ground field beyond its characteristic. The construction extends from \(\mathrm{GL}_n\) to group schemes over finite fields that are connected affine or of finite type.
\end{abstract}

\maketitle

\section*{Introduction}

The Lange--Stuhler theorem asserts that a vector bundle \(E\)
on a normal connected variety \(X\) over a perfect field \(k\)
of characteristic \(p>0\) is trivialized by a finite \'etale cover
if it is Frobenius-periodic, that is,
\(
F_X^{n,*}E\simeq E,\text{ for some }n\geq1.
\)
See \cite[Prop.~4.1.1]{Katz73} and \cite{LS77}.
Deligne extended this to a correspondence between Frobenius-invariant principal
$G$-bundles and $G(\F_q)$-bundles over $X$, for a connected algebraic group
$G/\F_q$. This is recalled in \cite[\S 3.2]{Laszlo01} (see also \cite[\S 1]{BD07}).

In this note, we characterize Deligne's correspondence as the 2-universal property
of a fiber product of gerbes. This yields a short proof valid for arbitrary
categories fibered in groupoids over any field of characteristic $p$. Thus we
need neither the geometric assumptions of the classical formulation nor the
assumption that the ground field is perfect. The main geometric input is the generalized Lang theorem (Lem.~\ref{lem:lang}).

Fix a prime power \(q=p^a\). All fibered categories below
are fibered in groupoids. We work with fpqc classifying stacks and
separated group schemes. A \(G\)-torsor over \(\sX\) means a
1-morphism \(\sX\to BG\). For an abstract group \(\Gamma\), write
\resizebox{3.5cm}{!}{
    $\Gamma_K^{\mathrm{const}}:=\coprod_{\gamma\in\Gamma}\Spec K$
}
for its constant group scheme over a field \(K\).
Write \(F_{\sX}\) for absolute \(p\)-Frobenius: on a scheme-valued
object \(x\in\sX(S)\), it is given by \(F_S^*x\).
This defines Frobenius on any category fibered in groupoids over
\(\Fp\), naturally in \(\sX\); see
\cite[Notations and Conventions]{TZ1}. For \(\sX/\Fq\),
\(F_{\sX}^a\) is an \(\Fq\)-morphism, and for \(G/\Fq\)
we have \(F_{BG}^{as}\simeq B(F_G^{as})\) for every \(s\geq1\).
Mapping categories below are groupoids, with \(2\)-isomorphisms
as arrows.

\Needspace{12\baselineskip}
\begin{thmI}\label{thm:main}
Let \(G\) be a connected group scheme over \(\Fq\), assumed
either affine or of finite type. For \(s\geq1\), set
\(\sigma\coloneqq F_G^{as}\) and
\(
H\coloneqq \Eq(\sigma,\id_G)\) the equalizer:
\[
H(S)=\{g\in G(S)\mid\sigma(g)=g\}\quad(\forall\, S/\Fq).
\]
For every category fibered in groupoids \(\sX/\Fq\), there is
a natural equivalence
\[
\Hom_{\Fq}(\sX,BH)\xrightarrow{\ \sim\ }
\left\{\,\vcenter{\hbox{\(
\xymatrix@C=1.5em@R=0.7em{
\sX\ar[rr]^{F_{\sX}^{as}}\ar[ddr]_{u}
  &&\sX\ar[ddl]^{u}\\
  &\scriptstyle\alpha\;\Leftarrow&\\
  &BG&
}
\)}}\,\right\}.
\]
The right side consists of triangles with a specified
\(2\)-isomorphism \(\alpha:uF_{\sX}^{as}\Rightarrow u\).
A morphism \((u,\alpha)\longrightarrow(u',\alpha')\)
is a \(2\)-isomorphism \(b:u\Rightarrow u'\) such that
\[
b\circ\alpha=\alpha'\circ(bF_{\sX}^{as}).
\]
The equivalence sends \(v\) to \(u=Bi\circ v\), where
\(i:H\hookrightarrow G\), with its canonical Frobenius structure.
The group scheme \(H\) is profinite \'etale; if \(G\) is of
finite type, \(H\) is finite \'etale and
\[
H\times_{\Fq}\mathbb F_{q^s}
\simeq G(\mathbb F_{q^s})_{\mathbb F_{q^s}}^{\mathrm{const}}.
\]
\end{thmI}

In particular, for \(s=1\) and finite-type \(G\),
\(H=G(\Fq)_{\Fq}^{\mathrm{const}}\).
There are no geometric hypotheses on \(\sX\): it may be any
fibered category over any field extension \(k/\Fq\).
The triangle is taken over \(\Fq\), since \(F_{\sX}^{as}\)
need not be \(k\)-linear. Taking \(q=p\) and \(G=\GL_r\)
gives Lange--Stuhler for arbitrary fibered categories over every
field of characteristic \(p\); see Cor.~\ref{cor:LS}.

\section{The cartesian square of gerbes}

Let \(K\) be a field, let \(G/K\) be a group scheme, and let
\(\sigma:G\to G\) be an endomorphism over \(K\).
Define the Lang morphism and its identity fiber by
\[
L_\sigma(g)=g\sigma(g)^{-1},\qquad
H=L_\sigma^{-1}(1)=\Eq(\sigma,\id_G).
\]
The fiber \(H\) is a subgroup scheme. In general \(L_\sigma\)
is not a group homomorphism, so no morphism \(B L_\sigma\) is
available. The following square replaces it.

\begin{prop}\label{prop:cartesian}
Assume that \(L_\sigma:G\to G\) is an fpqc covering.
Then the square
\[
\xymatrix@C=4em@R=2.5em{
BH \ar[r]^{Bi}\ar[d]_{Bi} & BG\ar[d]^{(\id,B\sigma)}\\
BG \ar[r]_{\Delta} & BG\times_K BG
}
\]
is \(2\)-cartesian, with the \(2\)-commutativity induced by
\(\sigma i=i\).
\end{prop}

\begin{proof}
The fiber product of the graph of \(B\sigma\) with the diagonal
is the stack \(\Fixst_{B\sigma}(BG)\) of pairs
\((P,\varphi:B\sigma(P)\xrightarrow{\sim}P)\).
Its distinguished object is
\[
e=(G,\mathrm{can}:B\sigma(G)\xrightarrow{\sim}G),
\]
where \(\mathrm{can}\) identifies the two trivial torsors.
Its automorphism group is \(H\). We show that every object is
fpqc locally isomorphic to \(e\).

Choose a local section \(z\) of the torsor underlying \(P\).
Its induced section of \(B\sigma(P)\) will be denoted by
\(z_\sigma\). Write \(\varphi(z_\sigma)=za\).
By assumption we can fpqc locally solve \(L_\sigma(g)=a\).
For \(z'=zg\), this gives
\[
\varphi(z'_\sigma)
=\varphi(z_\sigma\sigma(g))=za\sigma(g)=zg=z'.
\]
Thus the trivialization \(\tau:G\xrightarrow{\sim}P\),
\(h\mapsto z'h\), gives a commutative square
\[
\xymatrix@C=3.5em@R=2em{
B\sigma(G)\ar[r]^{\mathrm{can}}\ar[d]_{B\sigma(\tau)}^{\sim}
 &G\ar[d]^{\tau}_{\sim}\\
B\sigma(P)\ar[r]_{\varphi}&P.
}
\]
It therefore identifies the pair \((P,\varphi)\) with \(e\).
The fixed-object stack is consequently the neutral gerbe with
distinguished object of automorphism group \(H\).
The resulting equivalence
\[
BH\xrightarrow{\ \sim\ }\Fixst_{B\sigma}(BG)
\]
has the required projections: each sends \(e\) to the trivial
\(G\)-torsor and induces \(i:H\hookrightarrow G\) on
automorphism groups. Both composites with \(BG\) are therefore
\(Bi\). The comparison at \(e\) is \(\mathrm{can}\), giving
the \(2\)-commutativity induced by \(\sigma i=i\).
\end{proof}

Under the fpqc-covering hypothesis of Prop.~\ref{prop:cartesian},
\(L_\sigma\) is a right \(H\)-torsor. Indeed, over every
test scheme, the equality
\(L_\sigma(g)=L_\sigma(g')\) holds if and only if
\(g'=gh\) for a unique \(h\in H\). Thus
\[
G\times_K H\xrightarrow{\ \sim\ }
G\times_{L_\sigma,G,L_\sigma}G,
\qquad(g,h)\longmapsto(g,gh).
\]
In particular the quotient sheaf \(G/H\) is identified with
\(G\) by Lang's morphism. This is also the base change of
Prop.~\ref{prop:cartesian} along the trivial torsor
\(\Spec K\to BG\) in the upper-right corner.

\begin{rmk}\label{rmk:commutative}
Under the same fpqc-covering hypothesis, if \(G\) is commutative,
\(L_\sigma\) is a homomorphism and
we have the familiar \(2\)-cartesian square
\[
\xymatrix@C=4em@R=2.5em{
BH\ar[r]^{Bi}\ar[d] & BG\ar[d]^{B L_\sigma}\\
\Spec K\ar[r] & BG.
}
\]
Here the lower map is the trivial torsor. This follows from the
exact sequence
\(
1\to H\to G
\xrightarrow{L_\sigma}G\to 1
\)
of fpqc-sheaves of groups.
The square in Prop.~\ref{prop:cartesian} expresses the same
fixed-object construction for arbitrary \(G\).
\end{rmk}

\section{Lang's morphism without smoothness}

\begin{lem}\label{lem:lang}
Let \(G/\Fq\) be connected and put \(\sigma=F_G^{as}\),
where \(s\geq1\).
If \(G\) is of finite type, \(L_\sigma\) is a finite \'etale
surjection, and \(H=\Eq(\sigma,\id_G)\) is finite \'etale.
If \(G\) is affine of arbitrary type, \(L_\sigma\) is an fpqc
covering and \(H\) is profinite \'etale.
\end{lem}

\begin{proof}
First suppose that \(G\) is of finite type. Since \(\Fq\)
is perfect, \(G_{\red}\) is a smooth connected subgroup.
Its identity section and \cite[\href{https://stacks.math.columbia.edu/tag/056R}{056R}]{stacks-project}
show that it is geometrically connected. Lang's theorem
\cite[Thm.~1 and Cor., p.~557]{Lan56} gives surjectivity of
\(L_\sigma\) on geometric points of \(G_{\red}\), hence of \(G\).

To prove \'etaleness without smoothness, consider a square-zero
closed immersion \(j:T_0\hookrightarrow T\) of affine
\(\Fq\)-schemes. Suppose that \(g_0:T_0\to G\) and
\(a:T\to G\) satisfy \(L_\sigma(g_0)=a|_{T_0}\).
Since \(q^s\geq2\), the \(q^s\)-power map kills the defining ideal
of \(T_0\); hence \(F_T^{as}=j\circ t\) for a morphism
\(t:T\to T_0\). Put \(b=g_0\circ t\) and \(g=ab\).
Then
\[
b|_{T_0}=\sigma(g_0),\qquad
 g|_{T_0}=a|_{T_0}\,\sigma(g_0)=g_0.
\]
Naturality of Frobenius gives
\[
\sigma(g)=g\circ F_T^{as}=g\circ j\circ t
 =g|_{T_0}\circ t=g_0\circ t=b.
\]
Consequently,
\[
L_\sigma(g)=g\sigma(g)^{-1}=gb^{-1}=a.
\]
Any other solution \(h\) has \(\sigma(h)=b\); the equation
\(L_\sigma(h)=a\) then forces \(h=ab=g\).
Thus \(L_\sigma\) is formally \'etale.
It is locally of finite presentation, and is therefore \'etale.
It is also quasi-compact, so its surjectivity makes it an fpqc covering.

Its identity fiber \(H\) is an \'etale scheme of finite type
over a field, hence finite \'etale. The torsor identity established
after Prop.~\ref{prop:cartesian} shows that \(L_\sigma\)
is an \(H\)-torsor; in particular it is finite. Its geometric
fixed points are exactly \(G(\mathbb F_{q^s})\), which proves
\(H_{\mathbb F_{q^s}}
\simeq G(\mathbb F_{q^s})_{\mathbb F_{q^s}}^{\mathrm{const}}\).

Now let \(G\) be affine. Write \(G=\varprojlim_iG_i\) as a
limit of connected finite-type affine quotient group schemes (cf.~\cite[Cor.~2.7]{DM82}). Then
\[
L_\sigma=\varprojlim_i L_{\sigma,i},\qquad
H=\varprojlim_i H_i.
\]Here the second identification follows because equalizers
commute with inverse limits and Frobenius is compatible
with the transition morphisms.
Each \(L_{\sigma,i}\) is faithfully flat and each \(H_i\)
is finite \'etale by the finite-type case. Filtered colimits of
faithfully flat ring maps are faithfully flat, so \(L_\sigma\)
is fpqc and \(H\) is profinite \'etale. 
\end{proof}

\begin{proof}[Proof of Theorem~{\hyperref[thm:main]{\ref*{thm:main}}}]
Naturality identifies \(B\sigma\circ u\) with
\(u\circ F_{\sX}^{as}\), so \(\alpha\) makes the outer cone
\(2\)-commutative in
\[
\xymatrix@C=3em@R=2.2em{
\sX\ar@{-->}[dr]^{v}\ar@/^1pc/[drr]^{u}
  \ar@/_1pc/[ddr]_{u}&&\\
 &BH\ar[r]^{Bi}\ar[d]_{Bi}&BG\ar[d]^{(\id,B\sigma)}\\
 &BG\ar[r]_{\Delta}&BG\times_{\Fq}BG.
}
\]
The square is \(2\)-cartesian by Lem.~\ref{lem:lang} and
Prop.~\ref{prop:cartesian}. Its universal property gives the
asserted equivalence, naturally in \(\sX\).
\end{proof}

\Needspace{20\baselineskip}
\section{The Lange--Stuhler corollary}

\begin{cor}\label{cor:LS}
Let \(k\) be any field of characteristic \(p>0\), let
\(\sX\) be any category fibered in groupoids over \(k\), and let \(E\) be
a vector bundle of constant rank \(r\). Suppose that
\(F_{\sX}^{n,*}E\simeq E\), with \(n\geq1\).
There is a finite \'etale \(k\)-group scheme
\[
H_{r,n,k}:=
\Eq(F_{\GL_{r,\Fp}}^n,\id)\times_{\Fp}k
\]
and an \(H_{r,n,k}\)-torsor \(u:\sX\to B_kH_{r,n,k}\)
such that \(E\simeq u^*V\), where \(V\) is the vector bundle
on \(B_kH_{r,n,k}\) corresponding to the inclusion-induced map
\(B_kH_{r,n,k}\to B_k\GL_r\).
The cover
\[
T=\sX\times_{B_kH_{r,n,k}}\Spec k\longrightarrow\sX
\]
is representable finite \'etale and surjective, and trivializes
\(E\). If \(k\supset\mathbb F_{p^n}\), in particular if \(k\) is
algebraically closed, one can take the finite constant group
\(H_{r,n,k}=\GL_r(\mathbb F_{p^n})_k^{\mathrm{const}}\).
\end{cor}

\begin{proof}
Choose an isomorphism \(\varphi:F_{\sX}^{n,*}E\to E\).
The bundle \(E\) defines a morphism
\(v:\sX\to B_{\Fp}\GL_r\), and \(\varphi\) is precisely
a \(2\)-isomorphism \(vF_{\sX}^n\Rightarrow v\).
Thm.~\ref{thm:main}, applied with \(q=p\) and \(s=n\), gives
a lift to the classifying stack of
\(\Eq(F_{\GL_r}^n,\id)\). Combining this lift with the
structure map \(\sX\to\Spec k\) gives \(u\).
Since the composite
\[
\sX\xrightarrow{\ u\ }B_kH_{r,n,k}
\xrightarrow{\ i_k\ } B_k\GL_r
\]
classifies \(E\), we have \(E\simeq u^*i_k^*W\), where \(W\) is the universal
bundle on \(B_k\GL_r\).
The pullback of \(E\) to \(T\) is trivial as \(T\to \sX\to B_kH_{r,n,k}\) factors \(k\), and the
assertions about \(T\) follow from the fact that the map \(\Spec k\to
B_kH_{r,n,k}\) is finite étale and surjective. The constant-group description
follows from Lem.~\ref{lem:lang}.
\end{proof}

\section*{Acknowledgements}
\small{I would like to thank A.~Langer and D.~Weißmann for helpful discussions.
The author was partially supported by the General Program of the National Natural Science Foundation of China (Grant No.~12671059) and the Guangdong Basic and Applied Basic Research Foundation (Grant No.~2025A1515012175). }

\printbibliography
\end{document}

%% file: packages_and_functions.for_preamble.tex
\usepackage{mathabx}
\usepackage{bbm}
\usepackage[T1]{fontenc}
\usepackage{amsfonts}
\usepackage{braket}
\usepackage{mathrsfs}
\usepackage{stmaryrd}
\usepackage{upgreek}
\usepackage{textcomp}

\usepackage{extarrow}
\usepackage{enumerate}
\usepackage{enumitem}
\usepackage{graphicx}
\usepackage[colorlinks=true, linkcolor=blue]{hyperref}
\usepackage{mathtools}
\usepackage{mathtext}
\usepackage{amsmath}
\usepackage{amssymb}
\usepackage{amsthm}
\usepackage{tikz}
\usetikzlibrary{arrows.meta, positioning}
\usepackage[all,cmtip]{xy}
\usepackage[citestyle=alphabetic, backend=biber,
 bibstyle=alphabetic, giveninits=true, uniquelist = false, uniquename=init, isbn=false, maxcitenames=3, 
 maxbibnames=999, doi=false, url=false
 ]{biblatex}
\IfFileExists{Crystalline_crystals_references.bib}
  {\addbibresource{Crystalline_crystals_references.bib}}{}
\AtEveryBibitem{\clearlist{language}}

\usepackage[a4paper,margin=1.2in]{geometry}

\usepackage[colorinlistoftodos, textwidth=2.7cm]{todonotes}
\setuptodonotes{fancyline}
\usepackage{setspace} 

\usetikzlibrary{arrows}
\usetikzlibrary{matrix}
\usetikzlibrary{shapes}
\usetikzlibrary{decorations.pathmorphing,decorations.pathreplacing}
\usetikzlibrary{matrix}

\DeclareMathOperator{\GL}{\textup{GL}}

\DeclareMathOperator{\Spec}{\textup{Spec}}

%% file: packages_and_functions.tex
\global\long\def\F{\mathbb{F}}

\global\long\def\id{\textup{id}}

\newcommand{\sX}{{\mathcal X}}

\newcommand{\Hom}{{\rm Hom}}

\newcommand{\red}{{\rm red}}